\documentclass[11pt]{amsart}
\title{Equivariant Framed 1-Manifolds and the Pontryagin-Thom Isomorphism}
\author{Lucas Williams}
\address{Department of Mathematics and Statistics, Binghamton University, Binghamton, NY 13902, USA}
\email{lwilli39@binghamton.edu}

\usepackage{amsmath,amssymb,amsthm,array,color,multirow,pdflscape,mathrsfs,subcaption,todonotes}
\usepackage{tikz,tikz-cd}
\usetikzlibrary{bending}
\usepackage{float}
\usepackage{imakeidx}
\usepackage[margin=1.25in]{geometry}
\makeatletter
\def\l@section{\@tocline{1}{0pt}{1pc}{}{}}
\def\l@subsection{\@tocline{2}{0pt}{1pc}{4.6em}{}}
\def\l@subsubsection{\@tocline{3}{0pt}{1pc}{7.6em}{}}
\renewcommand{\tocsection}[3]{%
	\indentlabel{\@ifnotempty{#2}{\makebox[2.3em][l]{%
				\ignorespaces#1 #2.\hfill}}}#3}
\renewcommand{\tocsubsection}[3]{%
	\indentlabel{\@ifnotempty{#2}{\hspace*{2.3em}\makebox[2.3em][l]{%
				\ignorespaces#1 #2.\hfill}}}#3}
\renewcommand{\tocsubsubsection}[3]{%
	\indentlabel{\@ifnotempty{#2}{\hspace*{4.6em}\makebox[3em][l]{%
				\ignorespaces#1 #2.\hfill}}}#3}
\makeatother
\newcommand{\R}{\mathbb R}
\newcommand{\C}{\mathbb C}
\newcommand{\Z}{\mathbb{Z}}

\newcommand{\Sph}{\mathbb{S}}
\newcommand{\RP}{\mathbb{RP}}

\newcommand{\po}[2]{\ar@{}@<{#2}>[rd]|({#1})*\txt{\Large $\ulcorner$}}
\newcommand{\pb}[2]{\ar@{}@<{#2}>[rd]|({#1})*\txt{\Large $\lrcorner$}}

\usepackage[unicode=true, pdfusetitle,
 bookmarks=true,bookmarksnumbered=false,
 breaklinks=false,
 backref=false,
 colorlinks=true,
 linkcolor=blue,
 citecolor=blue,
 urlcolor=blue,
 final
]{hyperref}
\usepackage{amsthm}
\usepackage{aliascnt}
\usepackage{import}
\usepackage[capitalise]{cleveref}

\newcommand{\newrefformat}[2]{}

\newaliascnt{dia}{equation}
\aliascntresetthe{dia}



\crefname{dia}{Diagram}{Diagrams}
\creflabelformat{dia}{#2\textup{(#1)}#3}
\makeatletter
\newenvironment{diagram*}[1][]{%
    \begin{equation*}%
    \begin{tikzcd}[#1]%
    \refstepcounter{dia}
}{%
    \end{tikzcd}%
    \end{equation*}%
}
\makeatother

\crefname{dia}{Diagram}{Diagrams}
\creflabelformat{dia}{#2\textup{(#1)}#3}

\theoremstyle{plain}   % This is the default, anyway
\newtheorem{thm}{Theorem}[section] % numbered theorem
\makeatletter\let\c@thm\c@thm\makeatother

\makeatletter\let\c@cor\c@thm\makeatother
\newtheorem{lem}[thm]{Lemma}
\makeatletter\let\c@lemma\c@thm\makeatother
\newtheorem{prop}[thm]{Proposition}
\makeatletter\let\c@prop\c@thm\makeatother

\makeatletter\let\c@claim\c@thm\makeatother

\theoremstyle{definition}

\newtheorem{df}[thm]{Definition}
\makeatletter\let\c@defn\c@thm\makeatother

\makeatletter\let\c@const\c@thm\makeatother
\newtheorem{notn}[thm]{Notation}
\makeatletter\let\c@notn\c@thm\makeatother

\makeatletter\let\c@outline\c@thm\makeatother

\makeatletter\let\c@propty\c@thm\makeatother

\makeatletter\let\c@problem\c@thm\makeatother

\makeatletter\let\c@conj\c@thm\makeatother

\theoremstyle{remark}

\newtheorem{rmk}[thm]{Remark}
\makeatletter\let\c@rem\c@thm\makeatother
\newtheorem{ex}[thm]{Example}
\makeatletter\let\c@ex\c@thm\makeatother

\makeatletter\let\c@observationn\c@thm\makeatother

\makeatletter
\let\c@equation\c@thm
\numberwithin{equation}{section}
\makeatother

\crefname{lemma}{Lemma}{Lemmas}
\crefname{thm}{Theorem}{Theorems}
\crefname{defn}{Definition}{Definitions}
\crefname{notn}{Notation}{Notations}
\crefname{const}{Construction}{Constructions}
\crefname{prop}{Proposition}{Propositions}
\crefname{rem}{Remark}{Remarks}
\crefname{cor}{Corollary}{Corollaries}
\crefname{equation}{Equation}{Equations}
\crefname{ex}{Example}{Examples}
\crefname{propty}{Property}{Properties}
\crefname{problem}{Problem}{Problems}

\hypersetup{linktoc=all}
\let\originalleft\left
\let\originalright\right
\renewcommand{\left}{\mathopen{}\mathclose\bgroup\originalleft}
\renewcommand{\right}{\aftergroup\egroup\originalright}

\usepackage{tabularx}

\begin{document}

\begin{abstract}
	The Pontryagin-Thom theorem gives an isomorphism between the cobordism group of framed $n$-dimensional manifolds, $\omega_n$, and the $n^{th}$ stable homotopy group of the sphere spectrum, $\pi_n(\Sph)$. The equivariant analogue of this theorem, gives an isomorphism between the equivariant cobordism group of $V$-framed $G$-manifolds, $\omega_V^G$, and the $V^{th}$ equivariant stable homotopy group of the $G$-sphere spectrum, $\pi_V^G(\Sph)$, for a finite group $G$ and a $G$-representation, $V$. In this paper, we explicitly identify the images of each element of $\omega_1^{C_2}$ and $\omega_\sigma^{C_2}$ in $\pi_1^{C_2}(\Sph)$ and $\pi_\sigma^{C_2}(\Sph)$ under the equivariant Pontryagin-Thom isomorphism.  
\end{abstract}

\maketitle

\setcounter{tocdepth}{2}
\tableofcontents

\parskip 2ex

\section{Introduction}

A framed manifold is a closed compact smooth manifold equipped with a trivialization of its stable tangent bundle. The Pontryagin-Thom isomorphism \cite{thom, milnor_diff_top} is an isomorphism between the cobordism group of framed $n$-manifolds equipped with a reference map to $X$, denoted $\omega_n(X)$, and the $n^{th}$ stable homotopy group of the suspension spectrum of $X_+$, denoted $\pi_n(\Sigma^\infty_+ X)$. This isomorphism was an early example of the efficacy of applying homotopical techniques to differential topology and paved the way for such seminal work as \cite{ kervaire_milnor, luck_surgery, hhr, wang_xu_61stem}. 

Throughout this paper, we will write $\omega_n$ to mean $\omega_n(*)$. After producing an isomorphism $\omega_n \cong \pi_n(\Sph)$, it is natural to choose generators and ask where each element of $\omega_n$ is sent in $\pi_n(\Sph)$. For instance, $\omega_0$ is generated, under disjoint unions, by a single point endowed with a positively oriented framing. This is mapped to a suspension of the identity map $S^0\to S^0$. The manifold $S^1$ endowed with its Lie group framing generates $\omega_1$, and is mapped to a suspension of the Hopf fibration under the isomorphism $\omega_1\cong\pi_1(\Sph)$. 

For $G$ a finite group, and $V$ a finite dimensional orthogonal real $G$-representation, the cobordism group of $V$-framed $G$-manifolds equipped with a map to a fixed $G$-space $X$ is denoted $\omega_V^G(X)$. These groups, along with suspension isomorphisms, form an $RO(G)$-graded homology theory in the sense of \cite{alaska}. The equivariant Pontryagin-Thom theorem, originally proved in \cite{Hausschild1974, kosniowski_bordism}, gives an isomorphism from $\omega_V^G(X)$ to the $V^{th}$ equivariant stable homotopy group of the suspension spectrum of $X_+$, denoted $\pi_V^G(\Sigma_+^\infty X)$. We will consistently refer to the treatment of this theorem appearing in \cite{williams_cobordism}. 

Observe that $\omega_0^G$ is the abelian group of finite $G$-sets endowed with positively or negatively oriented framings. So when $V=0$, the equivariant Pontryagin-Thom isomorphism recovers the fact that $\pi_0^G(\Sph)$ is isomorphic to the Burnside ring $A(G)$. The work at hand is focused on the 1-dimensional version of this result.

In this paper, we compute the image of each element of $\omega_V^{C_2}$ in $\pi_V^{C_2}(\Sph)$ under the equivariant Pontryagin-Thom isomorphism, for $V$ either the trivial $C_2$-representation or the sign representation. This computation differs from the non-equivariant setting in surprising and interesting ways. Having a clean geometric example should provide helpful data for those working in equivariant stable homotopy theory, equivariant algebraic $K$-theory, and motivic stable homotopy theory through its connection to $C_2$-equivariant homotopy theory.

\begin{notn}
Throughout this paper, we use $\sigma$ to denote the sign representation of $C_2$. The notation $n+k\sigma=\R^{n+k\sigma}$ denotes the $C_2$-representation formed by taking the direct sum of $n$ copies of the trivial representation and $k$ copies of the sign representation. For a $C_2$-representation $V$, use $D(V)$ and $S(V)$ to denote the unit disc and sphere in $V$. Use $S^V$ to denote the one point compactification of $V$. Observe that $S^V\cong S(V\oplus \R)$. 
\end{notn}

We will now mention the main results of this paper. Let $X$ be a space. Using the Atiyah-Hirzebruch spectral sequence, one can show that 
\[
\pi_1(\Sigma_+^\infty X) = H_0(X;\Z/2)\oplus H_1(X;\Z).
\]

Now identify 
\[
\pi_1^{C_2}(\Sph) \cong \pi_1(\Sph)\oplus H_0(BC_2;\Z/2)\oplus H_1(BC_2;\Z)\cong \Z/2^{\oplus 3}
\]
using the tom Dieck splitting and the above formula for $\pi_1$ of a suspension spectrum. We will use this decomposition of $\pi_1^{C_2}(\Sph)$ frequently throughout this paper. 

Every $\R$-framed $C_2$-manifold is a disjoint union of $C_2\times S^1$, $S^1$ (with trivial action), and $S(2\sigma)$ as seen in \cref{ex:framing}. Thus, it suffices to say where these manifolds, equipped with their various framings, are sent by the equivariant Pontryagin-Thom map. As will be made precise later, we embed these $C_2$-manifolds in a $C_2$-representation and define its number of ``framing twists" to be the number of times a trivialization of the normal bundle twists the fibers as we traverse the circle. Our first theorem is as follows. 

\begin{thm}\label{thm:theorem 1}
The equivariant Pontryagin-Thom isomorphism sends the elements of $\omega_1^{C_2}$ to $\pi_1^{C_2}(\Sph)\cong \Z/2^{\oplus 3}$ as depicted in the following table:

\begin{center}
\begin{tabular}{ |c | c | c | c| }
 \hline
  & $\pi_1(\Sph)$ & $H_0(BC_2;\Z/2)$ & $H_1(BC_2;\Z)$ \\
 \hline 
 $S^1$ & number of framing twists & 0 & 0 \\
 \hline 
 $C_2\times S^1$ & 0 & number of framing twists on $S^1$ & 0 \\ 
 \hline 
 $S(2\sigma)$ & 0 & $(\text{number of framing twists})+1$ & 1 \\
 \hline  
\end{tabular}
\end{center} 
\end{thm}

We will define precisely what we mean by framing twists later in the paper.

Our second theorem concerns the $C_2$-equivariant stable stem indexed by the sign representation. It is not hard to compute that $\pi_\sigma^{C_2}(\Sph)\cong \Z$ by using the homotopy cofiber sequence
\[
{C_2}_+\to S^0\to S^\sigma.
\] 
As shown in \cite{morel, dugger_isaksen}, this group is generated by the equivariant Hopf map 
\[
S^{1+2\sigma}\to S^{1+\sigma}.
\]
This is of particular interest as the non-equivariant Hopf map is stably of order 2, while its equivariant analogue is of infinite order. We will give an interpretation of this notable difference between equivariant and non-equivariant homotopy theory in terms of framed manifolds at the end of the paper. 

Any $\sigma$-framed $C_2$-manifold is a disjoint union of $C_2\times S^1$ and $S(1+\sigma)$ as seen in \cref{ex:framing}.

\begin{thm}\label{thm:theorem 2}
Under the equivariant Pontryagin-Thom isomorphism $\omega_\sigma^{C_2}\to \pi_\sigma^{C_2}(\Sph)$, the manifold $C_2\times S^1$ is sent to the trivial element. The manifold $S(1+\sigma)$ is sent to either 0 or 1 in $\Z\cong\pi_\sigma^{C_2}(\Sph)$ depending on the parity of the number of framing twists on $S^1$.
\end{thm}

\begin{rmk}
In \cite{mcginnis_thesis}, McGinnis develops relations in the cobordism groups $\omega_V^{C_2}$ and the $C_2$-cobordism ring mirroring those in the $RO(C_2)$-graded stable stems. Combining these relations with the results of this paper would help to illuminate more of the structure of the $C_2$-equivariant framed cobordism groups and consequently, the equivariant stable stems. 
\end{rmk}

\subsection{Organization}

In section 2, we give the relevant background material on equivariant framed cobordism, and the Pontryagin-Thom isomorphism. In section 3, we discuss converting trivializations of tangent bundles to those of normal bundles. In section 4, we prove \cref{thm:theorem 1} giving framed manifold generators for $\pi_1^{C_2}(\Sph)$. In section 5, we prove \cref{thm:theorem 2} giving framed manifold generators of $\pi_\sigma^{C_2}(\Sph)$. 

\subsection{Acknowledgments}

The author thanks Cary Malkiewich, Tommy Brazelton, David Chan, Jesse Keyes, David Mehrle, Ben Spitz, and J.D. Quigley for helpful conversations and their encouragement to pursue this project. The author would also like to thank the anonymous referee for their helpful feedback. The author was partially supported by NSF FRG grant DMS-2052923. This work represents a portion of the author's PhD thesis written under the supervision of Cary Malkiewich at Binghamton University. 

\section{Preliminaries}

\subsection{Equivariant framed cobordism}

In this section, we give the necessary background to prove the main theorems.

\begin{df}
A $C_2$-manifold is a smooth compact manifold equipped with a smooth action of $C_2$. 
\end{df}

Given a $C_2$-manifold, $M$, the action on the tangent space is given by $g\cdot(x,v) = (gx,dg(v))$ where $dg$ is the map induced on the tangent space by the action of $g\in C_2$. 

\begin{df}\label{df:framing}
Let $M$ be a $C_2$-manifold and $V$ a real orthogonal $C_2$-representation. A $V$-framing of $M$ is an equivalence class of $C_2$-equivariant vector bundle isomorphisms
\[
TM\oplus \left(M\times \R^k\right)\cong M\times \left(V\oplus\R^k\right).
\]
We call two such isomorphisms equivalent if they are $C_2$-homotopic. Moreover, we say that 
\[
TM\oplus \left(M\times \R^{k+1}\right)\cong M\times \left(V\oplus\R^{k+1}\right)
\]
is equivalent to 
\[
TM\oplus \left(M\times \R^k\right)\cong M\times \left(V\oplus\R^k\right)
\]
if it is obtained by extending to the identity in the $\left(k+1\right)^{\text{st}}$ coordinate. 
\end{df}

In the above definition, we allow stabilization by trivial $C_2$-representations rather than arbitrary $C_2$-representations. This gives rise to the cobordism theory represented by the sphere spectrum and is therefore the definition we will use in this paper.

\begin{ex}\label{ex:framing}
The unit sphere in $\R^{2\sigma}$, denoted $S(2\sigma)$, is $\R$-framed. If $g$ is the nontrivial element of $C_2$ and $(x,v)\in TS(2\sigma)$ then $g\cdot(x,v)=(-x,v)$ so that $TS(2\sigma)\cong S(2\sigma)\times \R$. On the other hand, $S(1+\sigma)\cong S^\sigma$ is $\sigma$-framed. Model $S(1+\sigma)$ as the unit sphere in $\C$ with $C_2$-action given by complex conjugation. If $g$ is the nontrivial element of $C_2$ and $(z,v)\in TS(1+\sigma)$, then $g\cdot (z,v) = (\bar{z},-v)$ so that $TS(1+\sigma)\cong S(1+\sigma)\times \sigma$.
\end{ex}

\begin{df}
Let $M$ be a $C_2$-manifold and
\[
TM\oplus \left(M\times \R^k\right)\cong M\times \left(V\oplus\R^k\right)
\]
a $V$-framing of $M$. Use $-M$ to denote $M$ with the framing extended to
\[
TM\oplus \left(M\times\R^{k+1}\right)\cong M\times \left(V\oplus\R^{k+1}\right)
\]
by sending the $(k+1)^{\text{st}}$ coordinate to its negative.
\end{df}

\begin{df}
Two $V$-framed $C_2$-manifolds, $M$ and $N$, are cobordant if there exists a $(V\oplus\R)$-framed $C_2$-manifold, $W$, such that $\partial W \cong  M\amalg -N$ and the restriction of the framing on $W$ to its boundary induces the framing on $M\amalg -N$. The induced framing is given by pulling back the framing on the cobordism along the inclusion of the boundary.
\end{df}

\begin{ex}
As a $\sigma$-framed manifold, $C_2\times S^1$ with its Lie group framing is null-cobordant via the product of $S^1$ and the unit disk in $\sigma$. On the other hand, as we will prove later, $C_2\times S^1$ represents a non-trivial cobordism class as an $\R$-framed manifold. The null-cobordism we used before was specifically $(\sigma\oplus\R)$-framed. 
\end{ex}

\begin{rmk}
Denote the cobordism group of $\R$-framed $C_2$-manifolds as $\omega_1^{C_2}$. This group is isomorphic to $\Z/2^{\oplus 3}$ which may be computed either geometrically or homotopically although the geometric argument is more intricate than the non-equivariant version. We may choose a generating set of $\omega_1^{C_2}$ so that the underlying $C_2$-manifolds of the generators are $S^1$, $C_2\times S^1$, and $S(2\sigma)$. Denote cobordism group of $\sigma$-framed $C_2$-manifolds as $\omega_\sigma^{C_2}$ which is isomorphic to $\Z$. A generator may be chosen so that its underlying $C_2$-manifold is $S(1+\sigma)$. 
\end{rmk}

The following definition will be useful later for constructing specific framings. 

\begin{df}\label{df:orthogonal}
Let $V$ be a $C_2$-representation. Define $SO(V)$ to be the space of orientation preserving orthogonal transformations from $V$ to itself with $C_2$-action given by conjugation.
\end{df}

\subsection{The equivariant framed Pontryagin-Thom isomorphism}

We will give a short exposition of the equivariant Pontryagin-Thom construction tailored to the paper at hand. In what follows, let $V$ be $\R$ or $\sigma$. 

If $M\in \omega_V^{C_2}$, then $M$ embeds into $\R^{2+2\sigma}\oplus V$ so that the normal bundle of the embedding is $M\times \R^{2+2\sigma}$. A proof of a more general statement appears as \cite[Lemma 2.44]{williams_cobordism}.

The Pontryagin-Thom map, $\omega_V^{C_2}\to \pi_V^{C_2}(\Sph)$ is defined as follows. Let $M\in \omega_V^G$. Embed $M$ in  $\R^{2+2\sigma}\oplus V$ so that $\nu\left(M,\R^{2+2\sigma}\oplus V\right)\cong M\times \R^{2+2\sigma}$. Denote this normal bundle as $\nu$. Then the image of $M$ in $\pi_V^{C_2}(\Sph)$ is the following composition. 

\begin{align*}
S^{(2+2\sigma)\oplus V}
& \rightarrow D(\nu)/S(\nu)\\
& \rightarrow \left(M\times D\left(\R^{2+2\sigma}\right)\right)/\left(M\times S\left(\R^{2+2\sigma}\right)\right)\\
& \rightarrow D\left(\R^{2+2\sigma}\right)/S\left(\R^{2+2\sigma}\right) \cong S^{2+2\sigma}.
\end{align*} 
The first map is a Pontryagin-Thom collapse map, the second comes from the trivialization of $\nu$, and the third is induced by $M\to *$. If $M$ came equipped with an equivariant map to a $G$-space $X$, then the final map would be induced by $M\to X$ and the target would be $S^{2+2\sigma}\wedge_+ X$. As discussed in \cite{williams_cobordism}, this map is an isomorphism. 

\iffalse
\begin{rmk}
In the above construction, we chose a particular embedding of $M$ in $\R^{2+2\sigma}\oplus V$. Choosing a different embedding (with isomorphic normal bundle) would not affect the resulting class in $\pi_V^{C_2}(\Sph)$. Since we are working stably, we may increase the dimension of the ambient space of the embedding without changing the corresponding stable homotopy class in $\pi_V^{C_2}(\Sph)$. All embeddings of $M$ into a sufficiently high dimensional $C_2$-representation are equivariantly isotopic. Thus, our choice of embedding does not matter up to stabilization. 
\end{rmk}
\fi

\begin{rmk}
In the above construction, we chose a particular embedding of $M$ in $\R^{2+2\sigma}\oplus V$. Choosing a different embedding (with isomorphic normal bundle) does not affect the resulting class in $\pi_V^{C_2}(\Sph)$. Since we are working stably, we may increase the dimension of the ambient space of the embedding without changing the corresponding stable homotopy class in $\pi_V^{C_2}(\Sph)$. Any two equivariant embeddings $i,j\colon M \to W$, into some $C_2$-representation $W$, are equivariantly isotopic after stabilizing by $W$. This can be seen by concatenating the following equivariant isotopies:
\begin{align*}
M\times I & \rightarrow W\oplus W & M\times I &\rightarrow W\oplus W\\
(m,t)&\mapsto ((1-t)i(m),tj(m)) & (m,t)&\mapsto (tj(m),(1-t)j(m)).
\end{align*}
Thus, our choice of embedding does not matter up to stabilization. 
\end{rmk}

\section{Converting Tangent and Normal Bundle Trivializations}\label{subsec:framing}

In this subsection, we discuss compatibilities between trivializations of normal and tangent bundles. These compatibilities play an important role in the proofs in this paper. 

Recall that in \cref{df:framing} we defined a $V$-framing to be an equivariant trivialization of the stable tangent bundle
\begin{equation}\label{eq:framing}
TM\oplus\left(M\times\R^k\right)\cong M\times\left(V\oplus\R^k\right).
\end{equation}
However, the equivariant Pontryagin-Thom map takes as input a $V$-framed $C_2$-manifold equipped with a trivialization of its stable normal bundle. Given a $V$-framed $G$-manifold, we obtain a trivialization of its stable normal bundle by summing both sides of \cref{eq:framing} with the normal bundle. This is a major advantage of working with stable bundles as the same does not hold unstably. 

Unfortunately, there is not a bijective correspondence between equivariant trivializations of stable tangent and stable normal bundles because we only allow allow stabilization by trivial $G$-representations. Work of Waner suggests that allowing stabilization by non-trivial $G$-representations gives rise to the same $RO(G)$-graded cohomology theory despite the fact that these two notions of framing are not equivalent \cite{waner_1984_eq_bordism}. We plan to pursue this line of thinking in future work. 
 
\begin{df}
Let $M$ be a $V$-framed $C_2$-manifold and embed it in some $C_2$-representation, $W$. A trivialization of $TM$ and $\nu(M,W)$ are compatible if they sum to the canonical trivialization of $TW|_M$ (possibly after stabilization). 
\end{df}

Let $TS^1\cong S^1\times\R$ be the Lie group trivialization. Now trivialize
\begin{equation}\label{eq:first_triv}
TS^1\oplus (S^1\times \R) \cong S^1\times\R^2
\end{equation}
by extending by the identity in the last coordinate. Relative to this trivialization, any other trivialization of $TS^1\oplus (S^1\times \R)$ is given by the homotopy class of a map $S^1\to SO(2)$. The set of such maps can be identified with $\Z$ by taking the degree. 

We now compute which stable trivialization of $\nu\left(S^1,\R^3\right)\oplus \left(S^1\times \R^2\right)$ is induced by each of the $\Z$ many framings just described. Because we are only interested in stable trivializations, it suffices to say which trivializations of $\nu\left(S^1,\R^3\right)\oplus \left( S^1\times\R^2 \right)$ are induced by the degree 0 and 1 maps $S^1\to SO(2)$. This is because the degree $n$ map $S^1\to SO(2)$ and the degree $(n\mod 2)$ map become homotopic after stabilizing to a map $S^1\to SO(k)$ for $k\geq 3$.

We now define several trivializations of $\nu\left(S^1,\R^3\right)\oplus \left( S^1\times\R^2 \right)$. Embed $S^1$ as the unit circle in the $xy$-plane of $\R^3$ where we identify $0\in S^1$ with $(1,0,0)\in\R^{3}$ and $\pi/2$ with $(0,1,0)$. 

\begin{df}\label{df:std}
Define a map $\R^2\to \nu(S^1,\R^3)_\theta$ by sending the standard basis of $\R^2$ to the image, under the rotation by $\theta$ linear transformation, of the positive unit $z$ vector and the outward radial unit vector. This gives the standard trivialization of $\nu\left( S^1, \R^3\right)$.
\end{df}

Any other trivialization of $\nu\left( S^1,\R^3\right)$ is defined by a map $S^1\to SO(2)$. 

\begin{prop}
The trivialization of $TS^1\oplus (S^1\times\R)$ from \cref{eq:first_triv} induces the trivialization of $\nu\left( S^1,\R^3\right)$ given by applying the degree 1 map $S^1\to SO(2)$ to the standard trivialization. 
\end{prop}

\begin{proof}
In this case, it suffices to say which trivialization of $\nu \left( S^1,\R^3 \right)$ is induced by the Lie group trivialization of $TS^1$ prior to stabilizing.  

Begin by taking the canonical trivialization of $T\R^{3}|_{S^1}$. Now obtain a fiberwise homotopic trivialization by applying the transformation $F\colon T_\theta\R^{3}|_{S^1}\to T_\theta\R^{3}|_{S^1}$ where
\begin{equation}\label{eq:deformation}
F=\begin{bmatrix}
\cos^2(\theta) & \sin(\theta)\cos(\theta) & -\sin(\theta)\\
\sin(\theta)\cos(\theta) & \sin^2(\theta) & \cos(\theta)\\
\sin(\theta) & -\cos(\theta) & 0
\end{bmatrix}.
\end{equation}

This transformation rotates the vector $(0,0,1)$ in each fiber to the unit tangent vector of $S^1$ and orthogonally extends to the other two standard basis vectors of $T_\theta\R^3|_{S^1}$ as depicted in \cref{fig:tangent,fig:normal}. Restricting to the image of the unit $x$ and $y$ vectors gives a trivialization $\nu\left(S^1,\R^3\right)\cong S^1\times \R^2$.
 
\begin{figure}[h]
\minipage{0.48\textwidth}
  \includegraphics[width=\linewidth]{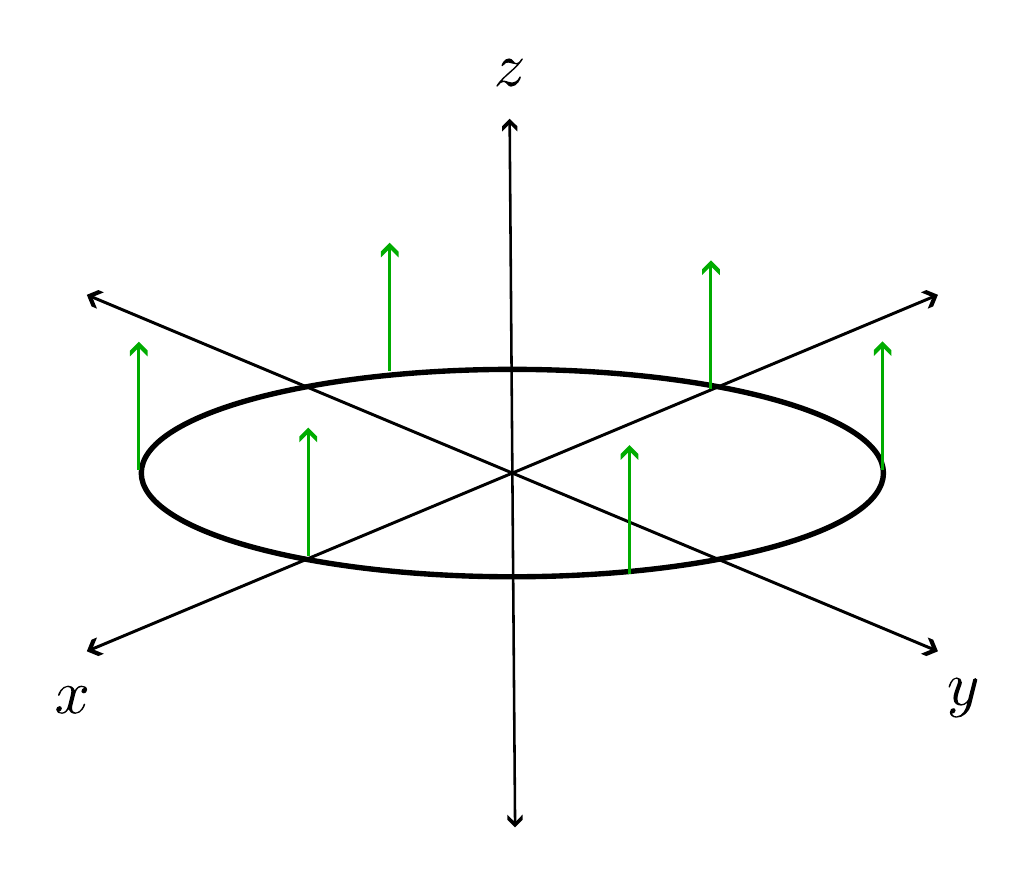}
  \captionsetup{labelsep=newline}
\endminipage\hfill
%\minipage{0.32\textwidth}
  %\includegraphics[width=\linewidth]{images/tangent1}
  %\captionsetup{labelsep=newline}
%\endminipage\hfill
\minipage{0.48\textwidth}%
  \includegraphics[width=\linewidth]{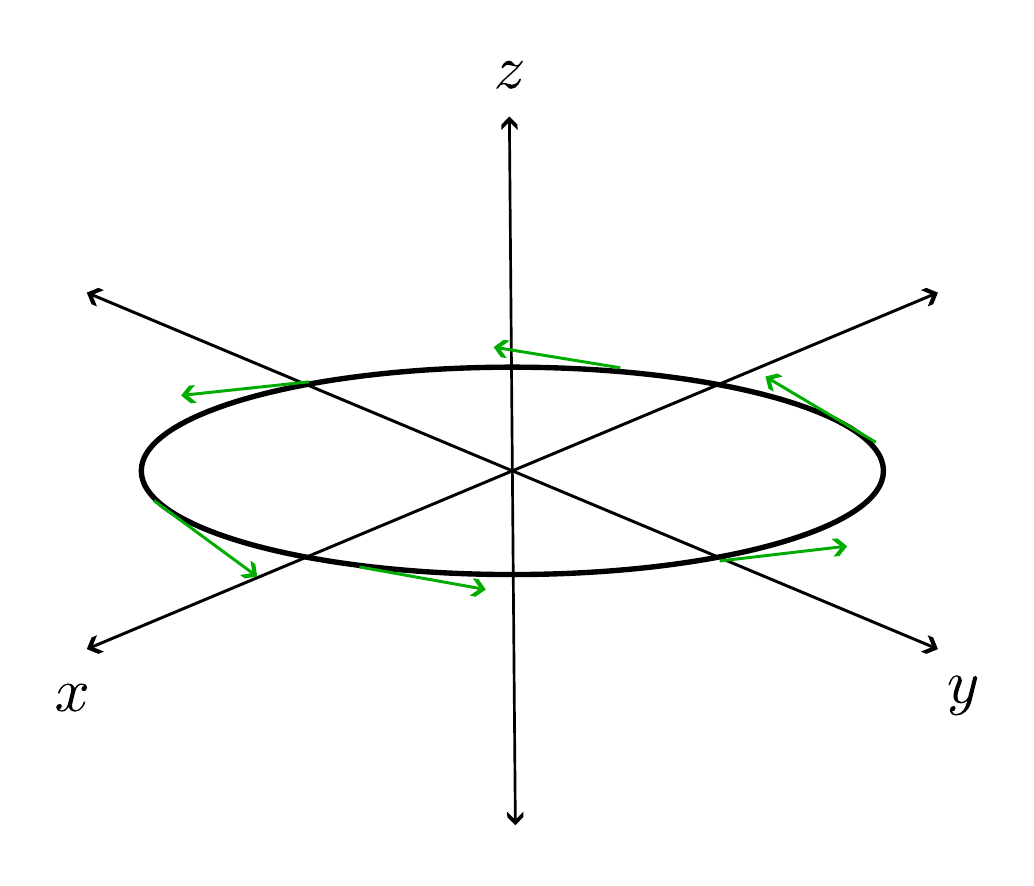}
  \captionsetup{labelsep=newline}  
\endminipage
\caption{The action of $F$ on the unit $z$ vectors in $T\R^{3}|_{S^1}$}\label{fig:tangent}
\end{figure}

\begin{figure}[h]
\minipage{0.48\textwidth}
  \includegraphics[width=\linewidth]{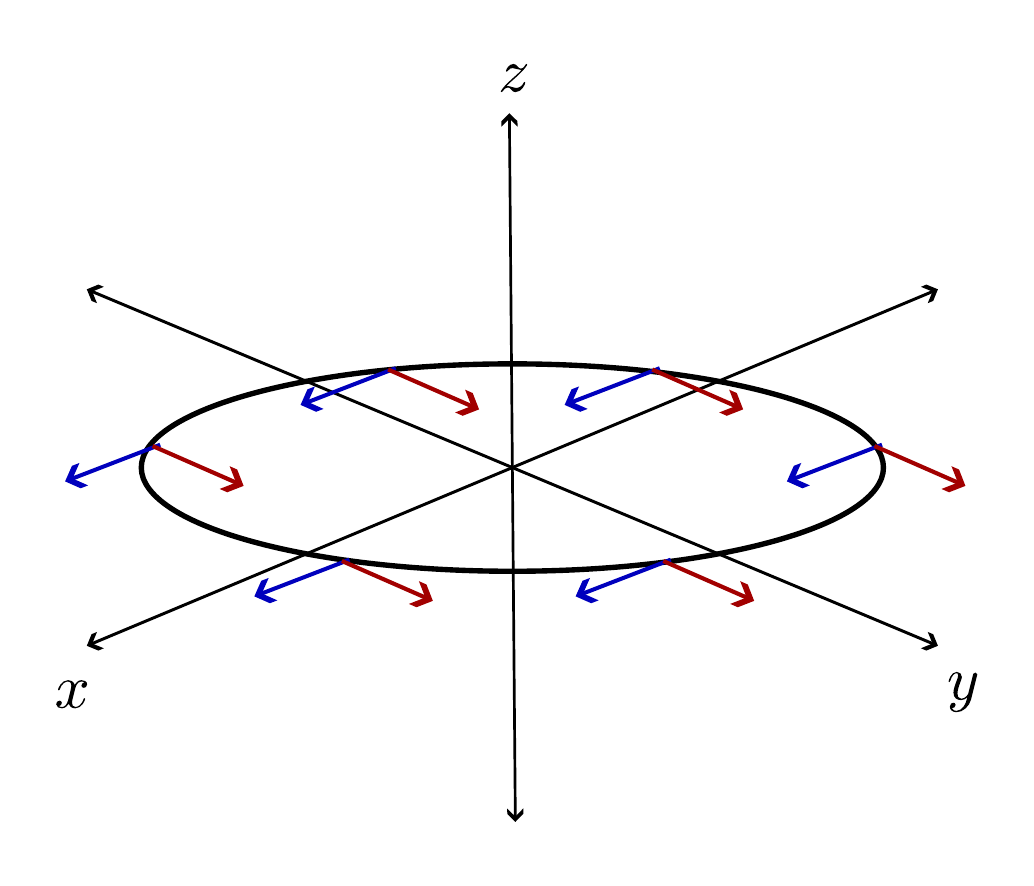}
  \captionsetup{labelsep=newline}
\endminipage\hfill\minipage{0.48\textwidth}
  \includegraphics[width=\linewidth]{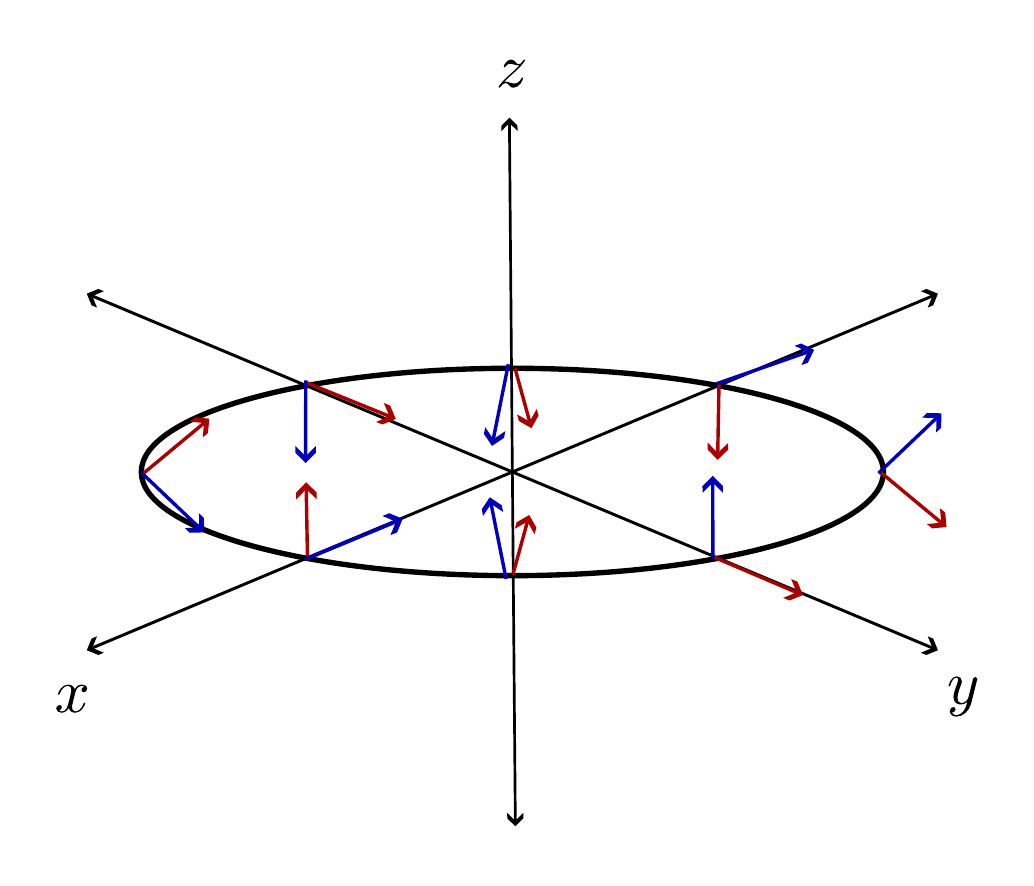}
  \captionsetup{labelsep=newline}
\endminipage
\caption{The action of $F$ on the unit $x$ and $y$ vectors in $T\R^{3}|_{S^1}$}\label{fig:normal}
\end{figure} 

Therefore, summing the Lie group framing of $S^1$ with the trivialization of $\nu\left(S^1,\R^3\right)$ described above gives rise to a map 
\[
T\R^3|_{S^1} \cong TS^1\oplus \nu\left(S^1,\R^3\right)\cong S^1\times\R^3
\]
which is fiberwise homotopic to the canonical trivialization $T\R^3|_{S^1}\cong S^1\times \R^3$ as desired. 
\end{proof}

\begin{prop}
The trivialization of $TS^1\oplus(S^1\times \R)$ defined by applying the degree 1 map $S^1\to SO(2)$ to \cref{eq:first_triv} induces the stabilization of the standard trivialization of $\nu\left( S^1,\R^3\right)\oplus \left(S^1\times\R^2\right)$ from \cref{df:std}.
\end{prop}

\begin{proof}
The trivialization of $\nu\left( S^1,\R^3\right)\oplus \left(S^1\times\R^2\right)$ in question is constructed by summing both sides of
\[
TS^1\oplus (S^1\times \R) \cong S^1\times \R^2,
\]
with $\nu\left(S^1,\R^3\right)$ to obtain 
\[
\nu\left( S^1,\R^3\right) \oplus (S^1\times \R^2) \cong \nu\left( S^1,\R^3\right)\oplus TS^1\oplus(S^1\times\R)\cong S^1\times \R^4.
\]
Relative to the stabilization of the standard trivialization of $\nu\left( S^1,\R^3\right)$, this trivialization may be defined by the homotopy class of a map $S^1\to SO(4)$. Let $\underline{\theta}$ and $I$ by the $2\times 2$ rotation by $\theta$ and identity matrices respectively. Then the map $S^1\to SO(4)$  giving this particular trivialization of $\nu\left( S^1,\R^3\right)\oplus (S^1\times \R^2)$ is defined by mapping $\theta\in S^1$ to $\left(\begin{array}{c|c} \underline{\theta} & 0\\ \hline 0 & \underline{\theta} \end{array} \right)$. Now observe that 
\[
\left(\begin{array}{c|c} \underline{\theta} & 0\\ \hline 0 & \underline{\theta} \end{array} \right) =\left(\begin{array}{c|c} \underline{\theta} & 0\\ \hline 0 & I \end{array} \right)\left(\begin{array}{c|c} 0 & I\\ \hline I & 0 \end{array} \right)\left(\begin{array}{c|c} \underline{\theta} & 0\\ \hline 0 & I \end{array} \right)\left(\begin{array}{c|c} 0 & I\\ \hline I & 0 \end{array} \right) = \left(\begin{array}{c|c} 0 & \underline{\theta}\\ \hline I & 0 \end{array}\right)^2,
\]
and that the map $S^1\to SO(4)$ given by sending $\theta$ to $\left(\begin{array}{c|c} 0&\underline{\theta} \\ \hline I&0 \end{array} \right)$ may be chosen as the generator of $\pi_1(SO(4))\cong \Z/2$. Since $SO(4)$ is a topological group, the multiplication in the group induces the operation on $\pi_1$ by the Eckmann-Hilton argument. So the map $\theta\mapsto \left(\begin{array}{c|c} \underline{\theta} & 0\\ \hline 0 & \underline{\theta} \end{array} \right)$ is the square of the generator and is thus, trivial. Therefore, the above trivialization of $\nu\left(S^1,\R^3\right) \oplus \left(S^1\times\R^2\right)$ is the standard trivialization as desired. 
\end{proof}

These compatibilities between trivializations will extend to equivariant settings throughout this paper. Now that we have explicit conversions of trivializations of tangent bundles to those of normal bundles, we will largely work with normal bundles. 

\section{The image of \texorpdfstring{$\omega_1^{C_2}$}{TEXT} in \texorpdfstring{$\pi_1^{C_2}(\Sph)$}{TEXT}}

In this section, we will compute where each element of $\omega_1^{C_2}$ is sent in $\pi_1^{C_2}(\Sph)$. 

As noted in the introduction, we will identify
\[
\pi_1^{C_2}\left(\Sph\right) \cong \pi_1\left(\Sph\right)\oplus \pi_1\left(\Sigma_+^\infty BC_2\right)\cong \pi_1\left(\Sph\right)\oplus H_0\left(BC_2;\Z/2\right)\oplus H_1\left(BC_2;\Z\right).
\]
The first isomorphism in the above equation is the tom Dieck splitting, and the second arises by observing that $\Sigma_+^\infty BC_2 \simeq \Sph\vee \Sigma^\infty BC_2$ and that $BC_2$ is connected. 

It is straightforward to show that every $C_2$-manifold admitting an $\R$-framing is a disjoint union of copies of $S^1$ with trivial action, $C_2\times S^1$, and $S(2\sigma)$.  

We begin by examining $\R$-framings of $S^1$. First embed $S^1$ as the unit sphere of the $xy$-plane in $\R^3$. We may then write a trivialization of the normal bundle 
\[
\nu\left(S^1,\R^3\right)\cong S^1\times\R^2
\]
by mapping the unit vector pointing in the outward radial direction in each fiber to the first standard basis vector, and the unit vector pointing in the positive $z$-direction to the second standard basis vector.

Relative to the above trivialization of $\nu\left(S^1,\R^3\right)$, any other $\R$-framing of $S^1$ is given by the homotopy class of a continuous map $S^1\to SO(2)$. We identify the group of all such maps with $\Z$ by taking the degree of a map $S^1\to SO(2)\simeq S^1$. We think of these framings as twisting the vectors of the framing discussed above an integer number of times as we traverse the fibers of the normal bundle of $S^1\hookrightarrow \R^3$. Note that this captures all possible stable framings because we stabilize only by trivial $G$-representations.

\begin{notn}
We denote $S^1$ equipped with the framing induced by the degree $n$ map $S^1\to SO(2)$ as $S^1_n$. With this labeling scheme, the explicit trivialization of $\nu\left(S^1,\R^3\right)$ discussed above is $S^1_0$. Recall from \cref{subsec:framing} that after translating trivializations of normal bundles to tangent bundles, $S^1_1$ is $S^1$ equipped with its Lie group framing. 
\end{notn} 

The following lemma will be helpful in each of the following computations.

\begin{lem}\label{lem:cobordism tom Dieck}
There is a ``tom Dieck splitting'' of the $\R$-framed $C_2$-cobordism group $\omega_1^{C_2}$ as $\omega_1\oplus \omega_1(BC_2)$. Moreover, the equivariant framed Pontryagin-Thom isomorphism is compatible with the tom Dieck splittings in equivariant framed cobordism and equivariant stable homotopy theory.
\end{lem}

\begin{proof}
The first statement is given by partitioning an $\R$-framed manifold into those connected components with a free $C_2$-action and those without. The second statement is precisely how the proof of the equivariant Pontryagin-Thom isomorphism for $\R$-framings is expressed in \cite{williams_cobordism}. 
\end{proof}

\begin{lem}\label{lem:split}
Identify,
\[
\omega_1\left(BC_2\right)\cong \pi_1\left(\Sigma_+^\infty BC_2\right)\cong H_0\left(BC_2;\Z/2\right)\oplus H_1\left(BC_2;\Z\right).
\]
The $H_0(BC_2;\Z/2)$ coordinate is controlled by the framing data of the 1-manifold. The $H_1(BC_2;\Z)$ coordinate is given by the homology class $\amalg_k S^1 \to BC_2\in \omega_1(BC_2)$. 
\end{lem}

\begin{proof}
The proof is immediate from the construction of the equivariant framed Pontryagin-Thom map and the fact that the splitting $\pi_1(\Sigma_+^\infty BC_2)$ is induced by $\Sigma_+^\infty BC_2\simeq \Sph\vee \Sigma^\infty BC_2$. 
\end{proof}

\begin{prop}
Under the identification
\[
\pi_1^{C_2}(\Sph) \cong \pi_1(\Sph)\oplus H_0(BC_2;\Z/2)\oplus H_1(BC_2;\Z)\cong\Z/2^{\oplus 3},
\]
the framed $C_2$-manifold $S^1_n$ corresponds to $(n,0,0)$.
\end{prop}

\begin{proof}
Since $S^1$ with the trivial action is not free, its image is entirely in the $\omega_1$ summand of $\omega_1^{C_2}\cong \omega_1\oplus \omega_1(BC_2)$. Thus, by \cref{lem:cobordism tom Dieck}, the image of $S^1$ under the equivariant Pontryagin-Thom isomorphism is entirely in the $\pi_1(\Sph)$ summand of the target. Since $C_2$ is acting trivially on $S^1$, the image of this framed manifold in the $\pi_1(\Sph)$ coordinate is precisely the same as in the non-equivariant setting. In other words, the $\pi_1(\Sph)$ coordinate is the degree of the map $S^1\to SO(2)$ (modulo 2) which induces the framing on $S^1_n$. 
\end{proof}

We now investigate framings of $C_2\times S^1$. First embed $C_2\times S^1$ in $\R^{2+\sigma}$ in the $(\R\oplus\sigma)$-plane. Write a trivialization 
\[
\nu\left(C_2\times S^1,\R^{2+\sigma}\right)\cong (C_2\times S^1)\times \R^{1+\sigma}
\]
by mapping the outward radial unit vector in each fiber to the $\sigma$-coordinate and the unit vector in the positive $z=\R$ direction to the $\R$-coordinate. Relative to this framing, any other framing of $C_2\times S^1$ is given by a $C_2$-equivariant map $C_2\times S^1\to SO(1+\sigma)$ with $SO(1+\sigma)$ as in \cref{df:orthogonal}. However, these are in bijective correspondence with non-equivariant maps $S^1\to SO(2)$ by restricting to one copy of $S^1$. We identify the set of such maps with $\Z$ by taking their degree. 

\begin{notn}
We denote $C_2\times S^1$ equipped with the framing induced by the degree $n$ map $S^1\to SO(2)$ by $C_2\times S^1_n$. 
\end{notn}

\begin{prop}\label{prop:swap}
Under the identification
\[
\pi_1^{C_2}(\Sph) \cong \pi_1(\Sph)\oplus H_0(BC_2;\Z/2)\oplus H_1(BC_2;\Z)\cong\Z/2^{\oplus 3},
\]
the framed manifold $C_2\times S^1_n$ corresponds to $(0,n,0)$.
\end{prop}

\begin{proof}
By \cref{lem:cobordism tom Dieck}, $C_2\times S^1$ is sent to 0 in the first coordinate of $\Z/2^{\oplus 3}$ as $C_2$ acts freely. 

Now model $EC_2$ as $S^\infty$ and $BC_2$ as $\RP^\infty$, the orbit space of $S^\infty$ by the antipodal action. The map $C_2\times S^1\to S^\infty$ is $C_2$-homotopic to the composite of the collapse $C^2 \times S^1 \to C_2 \times *$ and $C_2\times *\to S^\infty$ since the two components of the image can be $C_2$-equivariantly deformed to the north and south poles. Then after passing to orbits, the map $S^1\to \RP^\infty$ is null-homotopic. Thus, $C_2\times S^1$ is sent to 0 in the $H_1(BC_2;\Z)$ coordinate of $\Z/2^{\oplus 3}$. 

Observe that for any $G$-space $X$, the set of equivariant maps $C_2\times S^1\to X$, is in bijective correspondence with the set of non-equivariant maps $S^1\to X$. Consequently, a $C_2$-equivariant $\R$-framing of $C_2 \times S^1$ is equivalent to a non-equivariant framing of $S^1$. Therefore, the $H_0(BC_2;\Z/2)$ coordinate of $\Z/2^{\oplus 3}$ is the same as in the non-equivariant case --- it is given by $n$ (modulo 2).   
\end{proof}

Before discussing $S(2\sigma)$, the circle with antipodal action, we will need the following lemma. 

\begin{lem}\label{lem:orbits transfer}
The following diagram commutes:
\[
\begin{tikzcd}
\omega_1^{C_2}\left(EC_2\right) \arrow[d,"(-)/C_2"] \arrow[r] & \pi_1^{C_2}\left(\Sigma_+^\infty EC_2\right)\\
\omega_1\left(BC_2\right) \arrow[r] & \pi_1\left(\Sigma_+^\infty BC_2\right) \arrow[u,"tr"]
\end{tikzcd}.
\]
The horizontal maps are Pontryagin-Thom constructions, the left vertical map is taking $C_2$-orbits, and the right vertical map is the equivariant transfer. 
\end{lem}

\begin{proof}
The proof is a verification relying on the fact that the composition of Pontryagin-Thom collapse maps is a Pontryagin-Thom collapse map. A full proof may be found in \cite[Lemma 4.11]{periodic_points}. 
\end{proof}

We must now discuss $\R$-framings of $S(2\sigma)$, the 1-sphere equipped with the antipodal action. We will begin by enumerating trivializations of the normal bundle of $S(2\sigma)$. Embed $S(2\sigma)$ into $\R^{1+2\sigma}$ as the unit circle in the $2\sigma$-plane where we identify $2\sigma$ with the $xy$-plane. Begin with the canonical trivialization of $T\R^{1+2\sigma}|_{S(2\sigma)}\cong S(2\sigma)\times\R^{1+2\sigma}$. The transformation $T_\theta\R^{1+2\sigma}|_{S(2\sigma)}\to T_\theta\R^{1+2\sigma}|_{S(2\sigma)}$ defined in \cref{eq:deformation} is $C_2$-equivariant with respect to these choices (it may be helpful here to recall \cref{fig:tangent,fig:normal}). Then restricting to the image of the unit $x$ and $y$ vectors gives rise to a $C_2$-equivariant trivialization
\[
\nu\left(S(2\sigma),\R^{1+2\sigma}\right)\cong S(2\sigma)\times 2\sigma.
\]
After forgetting the $C_2$-action, the underlying framed manifold is $S^1_1$ discussed previously. 

Now any other trivialization of $\nu\left(S(2\sigma),\R^{1+2\sigma}\right)$ is given by the $C_2$-homotopy class of a $C_2$-equivariant map $S(2\sigma)\to SO(2\sigma)$ with $SO(2\sigma)$ as in \cref{df:orthogonal}. The action of $C_2$ on $SO(2\sigma)$ is given by conjugation and is thus, trivial. So an equivariant map $S(2\sigma)\to SO(2\sigma)$ must send antipodal points to the same element. Therefore, the set of homotopy classes of equivariant maps $S(2\sigma)\to SO(2\sigma)$ may be identified with $\Z$ by taking the degree which, importantly, must always be $2n$. We may think of a framing as ``twisting $n+\tfrac{1}{2}$ times on half of $S(2\sigma)$". Equivariance then dictates what must happen on the other half of the circle. 

\begin{notn}
We denote $S(2\sigma)$ equipped with the framing induced by the degree $2n$ map $S(2\sigma)\to SO(2\sigma)$ by $S(2\sigma)_n$.
\end{notn}

\begin{prop}\label{prop:antipodal_base}
Under the identification,
\[
\pi_1^{C_2}(\Sph)\cong \pi_1(\Sph)\oplus H_0(BC_2;\Z/2)\oplus H_1(BC_2;\Z),
\]
the framed manifold $S(2\sigma)_0$ defined above is mapped to $(0,1,1)$.
\end{prop}

\begin{proof}
Since $C_2$ acts freely on $S(2\sigma)$, this manifold is sent to $0$ in the $\pi_1(\Sph)$ component. 

Again model $EC_2$ as $S^\infty$ with the $C_2$-CW-complex structure with a $C_2$-2-cell in every dimension, and $BC_2$ as $\RP^\infty$. The element of $H_1(BC_2;\Z)$ represented by $S(2\sigma)$ is given by the homology class of the $C_2$-orbits of $S(2\sigma)\to EC_2$ as described in \cref{lem:split}. The map $S(2\sigma)\to S^\infty$  is the inclusion of the $1$-skeleton. After taking $C_2$-orbits, we obtain a map $S^1\to \RP^\infty$ which is again the inclusion of the $1$-skeleton. Thus, $S(2\sigma)$ is sent to $1$ in the $H_1(BC_2;\Z)$ summand. 

We now discuss the $H_0(BC_2;\Z/2)$ summand. By the discussion of \cref{subsec:framing}, the Lie group trivialization of $TS(2\sigma)$ gives rise to the framing on $S(2\sigma)_0$. This is because when these trivializations of $T(2\sigma)$ and $\nu\left(S(2\sigma),\R^{1+2\sigma}\right)$ are summed together, they give the canonical trivialization of $T\R^{1+2\sigma}|_{S(2\sigma)}$. So instead of working with $S(2\sigma)_0$ it suffices to work with $S(2\sigma)$ equipped with the Lie group trivialization of its tangent bundle. 

We now apply \cref{lem:orbits transfer}. Taking the $C_2$-orbits of $S(2\sigma)$ equipped with the Lie group trivialization of its tangent bundle gives $S^1$ with the Lie group trivialization of its tangent bundle as an element of $\omega_1(BC_2)$. Under the Pontryagin-Thom isomorphism, this is sent to a map $S^{n+1}\to S^n\wedge_+ BC_2$ which, when composed with $BC_2\to *$, is a suspension of the Hopf fibration. The projection map $\pi_1\left(\Sigma_+^\infty BC_2\right)\cong H_0(BC_2;\Z/2)\oplus H_1(BC_2;\Z) \to H_0(BC_2;\Z/2)$ is induced by $BC_2\to *$. So the image of $S(2\sigma)$ in $H_0(BC_2;\Z/2)\subset \pi_1(\Sigma_+^\infty BC_2)$ is 1. Since the equivariant transfer is an isomorphism which respects this splitting, we can see that the image of $S(2\sigma)$ equipped with this framing is $1$ in the $H_0(BC_2;\Z/2)$ summand, as desired. 
\end{proof}

We now compute the image of $S(2\sigma)$ equipped with any other framing.

\begin{prop}
Under the identification 
\[
\pi_1^{C_2}(\Sph)\cong \pi_1(\Sph)\oplus H_0(BC_2;\Z/2)\oplus H_1(BC_2;\Z)
\]
the framed $C_2$-manifold $S(2\sigma)_n$ is sent to $(0,n+1,1)$ by the equivariant Pontryagin-Thom isomorphism.
\end{prop}

\begin{proof}
First note that the $\pi_1(\Sph)$ and $H_1(BC_2;\Z)$ coordinates are not effected by changing the framing on $S(2\sigma)$, so the arguments of \cref{prop:antipodal_base} still hold. To show that the $H_0(BC_2;\Z/2)$ coordinate is $n+1$ we will show that there is an equivariant framed cobordism giving an equivalence
\[
S(2\sigma)_n \sim S(2\sigma)_0\amalg (C_2\times S^1_n).
\]
We first observe that there is a $C_2$-cobordism between the underlying manifolds. This can be constructed by taking $S(2\sigma)\times I$ and removing two open disks from opposite sides of the cylinder to create the $C_2\times S^1$ boundary component. Call this $C_2$-cobordism $M$. 

The product of the Lie group trivialization of $TS(2\sigma)$ and the canonical trivialization of $TI$ give rise to a $C_2$-equivariant trivialization $TM\cong M\times\R^2$. Embed $M$ as the unit sphere in $\R^{1+2\sigma}$ sitting inside of $\R^{2+2\sigma}$. We may convert this $\R^2$-framing of $M$ into a trivialization 
\[
\nu\left(M,\R^{2+2\sigma}\right)\oplus \left(M\times \R^2\right) \cong M\times\R^{2+2\sigma}.
\]

Relative to this trivialization, any other trivialization of $\nu\left(M,\R^{2+2\sigma}\right)$ is given by the equivariant homotopy class of an equivariant map $M\to SO(2+2\sigma)$. We have an equivariant map $\partial M\to SO(2\sigma)$ which extends to an equivariant map $M\to SO(2\sigma)$ by adding the degrees of the maps on boundary components together in a way that respects equivariance. To see this, note that $M$ is  $C_2$-homotopy equivalent to the wedge of three copies of $S^1$ with $C_2$-action restricting to the antipodal action on one copy of $S^1$ and to the action that interchanges the other two copies of $S^1$. Stabilizing this map by two copies of the trivial representation gives rise to a framing of $M$, the boundary of which is the disjoint union of $S(2\sigma)_n$ and $S(2\sigma)_0\amalg C_2\times S^1_n$. Applying \cref{prop:swap,prop:antipodal_base} gives the desired result.  
\end{proof}

\begin{rmk}
The results of this sections should hold more generally for $C_p$ any cyclic group of prime order. In this case we obtain a decomposition:
\[
\pi_1^{C_p}(\Sph)\cong \pi_1(\Sph)\oplus H_0(BC_p;\Z/2)\oplus H_1(BC_p;\Z) \cong \Z/2\oplus \Z/2\oplus \Z/p.
\]
The arguments concerning the $\pi_1(\Sph)$ and $H_0(BC_p;\Z/2)$ summands are nearly identical to those presented here. The $H_1(BC_p;\Z)$ summand presents slightly greater difficulty as the standard cell structure on $S^\infty$ is not compatible with free $C_p$ actions with $p\neq 2$. However, one may work around this by modeling $S^\infty$ as a colimit of odd dimensional spheres with free $C_p$-action. Then to compute the homology class of the image of $C_p\times S^1$ and $S^1$ with rotation by $2\pi/p$ action one may restrict their attention to the image in $S^3$. 
\end{rmk}

\section{The image of \texorpdfstring{$\omega_\sigma^{C_2}$}{TEXT} in \texorpdfstring{$\pi_\sigma^{C_2}(\Sph)$}{TEXT}}
In this section, we give explicit descriptions of the image of each $\sigma$-framed $C_2$-manifold under the equivariant Pontryagin-Thom isomorphism $\omega_\sigma^{C_2}\to \pi_\sigma^{C_2}(\Sph)$. Recall that $\sigma$ is the sign representation of $C_2$. 

We first observe that any $\sigma$-frameable $C_2$-manifold is a disjoint union of copies of $S(1+\sigma)$ and $C_2\times S^1$. We begin with the following proposition.

\begin{prop}
The $C_2$-manifold $C_2\times S^1$ equipped with any $\sigma$-framing is zero in $\omega_\sigma^{C_2}$. 
\end{prop}

\begin{proof}
First observe that $C_2\times S^1\in \omega_\sigma^{C_2}(EC_2)$. The homotopy cofiber sequence
\[
{C_2}_+\to S^0\to S^\sigma
\]
induces a long exact sequence
\[
\dots\to\omega_1^{C_2}(EC_2)\to \omega_1(EC_2)\to \omega_\sigma^{C_2}(EC_2) \to \omega_0^{C_2} (EC_2)\to \omega_0(EC_2). 
\]

The forgetful maps 
\[
\omega_0^{C_2}(EC_2)\to\omega_0(EC_2)\cong \omega_0 \qquad \text{and}\qquad \omega_1^{C_2}(EC_2)\to \omega_1(EC_2)\cong\omega_1
\]
are injective and surjective respectively. Thus, by exactness $\omega_\sigma^{C_2}(EC_2)$ is also zero. Then $C_2\times S^1$ is null-cobordant in $\omega_\sigma^{C_2}(EC_2)$ and thus also in $\omega_\sigma^{C_2}$. 
\end{proof}

We now examine $S(1+\sigma)$. We begin by classifying $\sigma$-framings of this manifold.

Embed $S(1+\sigma)$ in $\R^{1+2\sigma}$ as the unit circle in the $xy$-plane where we identify $x$ with the trivial $C_2$-representation. We now obtain a $C_2$-equivariant vector bundle isomorphism
\[
\nu\left(S(1+\sigma),\R^{1+2\sigma}\right)\cong S(1+\sigma)\times \R^{1+\sigma}
\]
by sending the unit vector pointing in the outward radial direction in each fiber to the first standard basis vector, and the unit vector pointing in the positive $z$-direction in each fiber to the second standard basis vector. This is our first $\sigma$-framing of $S(1+\sigma)$. Relative to this $\sigma$-framing, any other $\sigma$-framing is given by the $C_2$-homotopy class of a  $C_2$-equivariant map $S(1+\sigma)\to SO(1+\sigma)$ with $SO(1+\sigma)$ as in \cref{df:orthogonal}.

As a $2$-dimensional real vector space, $C_2$ acts on $\R^{1+\sigma}$ by $\begin{bmatrix}
1&0\\0&-1
\end{bmatrix}$. Thus, $C_2$ acts on $SO(1+\sigma)$ by 
\[
\begin{bmatrix}
\cos(\theta)&-\sin(\theta)\\\sin(\theta)&\cos(\theta)
\end{bmatrix}\mapsto \begin{bmatrix}\cos(\theta) &\sin(\theta)\\
-\sin(\theta) & \cos(\theta)
\end{bmatrix}.
\]

Note also that for our trivialization to be equivariant at the fixed points of $S(1+\sigma)$, the normal vector sent to $(1,0)$ in $\R^{1+\sigma}$ must come from a vector on which $C_2$ acts trivially. Thus, we may think of an equivariant map $S(1+\sigma)\to SO(1+\sigma)$ as follows. First choosing a path between the image of the two fixed points which winds around $SO(1+\sigma)$ a half integer number of times. Second take the negation of this path to get a full loop in $SO(1+\sigma)$. Note that by taking degrees of maps, the set of framings on $S(1+\sigma)$ is in bijection with $\Z$. 

\begin{notn}
We denote $S(1+\sigma)$ with the framing induced by the degree $n$ $C_2$-equivariant map $S(1+\sigma)\to SO(1+\sigma)$ as $S(1+\sigma)_n$.
\end{notn}

The equivariant Hopf fibration is a map 
\[
\eta\colon S^{1+2\sigma} \simeq \C^2\backslash \{0\} \rightarrow \mathbb{CP}^1\simeq S^{1+\sigma} 
\] 
given by restricting the projection map $\C^2\to \mathbb{CP}^1$ to the unit sphere. The $C_2$-action is given by complex conjugation. The equivariant Hopf fibration can be given by the formula $(z_0,z_1)\mapsto (2z_0\bar{z_1},|z_0|^2-|z_1|^2)$ which is $C_2$-equivariant. A straightforward calculation, proceeding from the homotopy cofiber sequence
\[
{C_2}_+\to S^0\to S^\sigma,
\] 
shows that $\pi_\sigma^{C_2}(\Sph)\cong \Z$. Then a theorem of Morel \cite{morel} which appears as \cite[Theorem 1.2]{dugger_isaksen} shows that $\eta$ generates $\pi_\sigma^{C_2}(\Sph)$. 

\begin{prop}
The $C_2$-manifold $S(1+\sigma)$ with its Lie group framing, is sent to the equivariant Hopf fibration $\eta\colon S^{1+2\sigma}\to S^{1+\sigma}$ by the equivariant Pontryagin-Thom isomorphism. 
\end{prop}

\begin{proof}
The Lie group framing of $S(1+\sigma)$ is compatible with the trivialization
\[
\nu\left(S(1+\sigma),\R^{1+2\sigma}\right)\cong S(1+\sigma)\times \R^{1+\sigma}
\]
coming from $S(1+\sigma)_1$. The argument is the same as in the non-equivariant case which is discussed in \cref{subsec:framing}.  This works out precisely because we chose our ambient space to be $\R^{1+2\sigma}$ rather than $\R^{2+\sigma}$. In other words, we stabilize by $\sigma$ rather than $\R$.  

The underlying framed manifold of $S(1+\sigma)_1$ is sent to the Hopf fibration by the non-equivariant Pontryagin-Thom map. This extends to the $C_2$-equivariant case when the Hopf fibration is written as a map $\C^2\backslash \{0\}\to \mathbb{CP}^1$ with $C_2$-action given by complex conjugation. 
\end{proof}

\begin{thm}
The image, under the Pontryagin-Thom isomorphism, of the $\sigma$-framed $C_2$-manifold $S(1+\sigma)_n$ in $\Z\cong \pi_\sigma^{C_2}(\Sph)$ is either $0$ if $n$ is even or $1$ if $n$ is odd.
\end{thm}

\begin{proof}
We previously showed that $S(1+\sigma)_1$ is sent to $1\in \Z\cong \pi_\sigma^{C_2}(\Sph)$. Here we will use the map $\omega_\sigma^{C_2}\to \omega_0$ given by taking the $C_2$-fixed points of a $\sigma$-framed $C_2$-manifold. 

This map takes $S(1+\sigma)_1$ to $2\in\Z\cong \omega_0$. So when written as a map $\Z\to \Z$ this is the multiplication by 2 map, which is an injection. In general, if $n$ is odd, then $S(1+\sigma)_n$ is sent to $2$ in $\Z\cong \omega_0$. Thus, if $n$ is odd, $S(1+\sigma)_n$ is cobordant to $S(1+\sigma)_1$. If, on the other hand, $n$ is even, then $S(1+\sigma)_n$ is mapped to $0\in\Z\cong \omega_0$. Thus, $S(1+\sigma)_n$ represents the trivial cobordism class in $\omega_\sigma^{C_2}$. 
\end{proof}

We have now shown where any element of $\omega_\sigma^{C_2}$ is sent in $\pi_\sigma^{C_2}(\Sph)$ by the Pontryagin-Thom isomorphism. Note that while a connected $\sigma$-framed $C_2$-manifold can only be mapped to $0$ or $1$, the disjoint union of $n$ copies of $S(1+\sigma)_1$ is sent to $n\in \Z$.

We end with some remarks on why the equivariant Hopf fibration is of infinite order rather than of order 2 as it is in the non-equivariant case. There is a $C_2$-equivariant cobordism from $S(1+\sigma)\amalg S(1+\sigma)$ to $S(1+\sigma)$ given by the involution on the pair of pants that reflects the front to the back. However, this only extends to a framed cobordism when the boundary components are $S(1+\sigma)_n\amalg S(1+\sigma)_n$ and $S(1+\sigma)_{2n}$ for $n$ even. In contrast, when $n$ is odd, $S(1+\sigma)_n\amalg S(1+\sigma)_n$ is not framed cobordant to a single copy of $S(1+\sigma)$. Thus, not every element of $\omega_\sigma^{C_2}$ is represented by a connected manifold. This difference between equivariant and non-equivariant framed cobordism is precisely why the equivariant Hopf map is of infinite order while its non-equivariant counterpart is of finite order.

%\printindex
\bibliographystyle{amsalpha}
\bibliography{streamlined}

\end{document}